\documentclass[12pt,reqno]{amsart}

\usepackage{amsmath,amsthm,amssymb,comment,fullpage}

\usepackage{times}

\usepackage[T1]{fontenc}

\usepackage{mathrsfs}

\usepackage{latexsym}

\usepackage[dvips]{graphics}

\usepackage{epsfig}

\usepackage{amsmath,amsfonts,amsthm,amssymb,amscd}

\input amssym.def

\input amssym.tex

\usepackage{color}

\usepackage{hyperref}

\usepackage{url}

\usepackage{breakurl}

\newcommand{\bburl}[1]{\textcolor{blue}{\url{#1}}}

\newcommand{\fix}[1]{\textcolor{red}{\textbf{\large (#1)\normalsize}}}

\numberwithin{equation}{section}

\newtheorem{thm}{Theorem}[section]

\newtheorem{cor}[thm]{Corollary}

\newtheorem{lem}[thm]{Lemma}

\theoremstyle{plain}

\newtheorem{definition}[thm]{Definition}

\newtheorem*{theorem*}{Theorem}

\newtheorem{remark}[thm]{Remark}

\newcommand\be{\begin{equation}}

\newcommand\ee{\end{equation}}

\newcommand\bea{\begin{eqnarray}}

\newcommand\eea{\end{eqnarray}}

\newcommand\bi{\begin{itemize}}

\newcommand\ei{\end{itemize}}

\newcommand\ben{\begin{enumerate}}

\newcommand\een{\end{enumerate}}

\newcommand\bc{\begin{center}}

\newcommand\ec{\end{center}}

\newcommand\ba{\begin{array}}

\newcommand\ea{\end{array}}

\newcommand{\R}{\ensuremath{\mathbb{R}}}

\newcommand{\C}{\ensuremath{\mathbb{C}}}

\newcommand{\Z}{\ensuremath{\mathbb{Z}}}

\newcommand{\N}{\mathbb{N}}

\newcommand\frakfamily{\usefont{U}{yfrak}{m}{n}}

\DeclareTextFontCommand{\textfrak}{\frakfamily}

\newtheorem{rek}[thm]{Remark}

\newcommand{\hr}[1]{\href{#1}{\url{#1}}}

\newcommand\erf{\operatorname{erf}}

\newcommand\stab{\operatorname{stab}}

\newcommand\fs{\mathfrak{s}}

\newcommand\fn{\mathfrak{n}}

\newcommand\fg{\mathfrak{g}}

\newcommand\fl{\mathfrak{l}}

\newcommand\fd{\mathfrak{d}}

\newcommand\fu{\mathfrak{u}}

\newcommand\fh{\mathfrak{h}}

\newcommand\ad{\operatorname{ad}}

\newcommand\Ad{\operatorname{Ad}}

\newcommand\spans{\operatorname{sp}}

\title{Leading Digit Laws on Linear Lie Groups} 

\author{Corey Manack}
\email{\textcolor{blue}{\href{mailto:cmanack@fandm.edu}{cmanack@fandm.edu}}}
\address{Department of Mathematics, Franklin \& Marshall, Lancaster, PA 17604}

\author{Steven J. Miller}
\email{\textcolor{blue}{\href{mailto:sjm1@williams.edu, Steven.Miller.MC.96@aya.yale.edu}{sjm1@williams.edu, Steven.Miller.MC.96@aya.yale.edu}}}
\address{Department of Mathematics and Statistics, Williams College, Williamstown, MA 01267}

\subjclass[2010]{	11K06, 60F99  (primary), 28C10, 15B52, 	15B99    (secondary).} 

\keywords{Benford's law, digit laws, Haar measure, matrix groups}

\thanks{The second named author was supported by NSF Grant DMS1265673.}

\date{\today}

\begin{document}

\maketitle

\begin{abstract} We determine the leading digit laws for the matrix components of a linear Lie group $G$. These laws generalize the observations that the normalized Haar measure of the Lie group $\R^+$ is $dx/x$ and that the scale invariance of $dx/x$ implies the distribution of the digits follow Benford's law, which is the probability of observing a significand base $B$ of at most $s$ is $\log_B(s)$; thus the first digit is $d$ with probability $\log_B(1 + 1/d)$). Viewing this scale invariance as left invariance of Haar measure, we determine the power laws in significands from one matrix component of various such $G$. We also determine the leading digit distribution of a fixed number of components of a unit sphere, and find periodic behavior when the dimension of the sphere tends to infinity in a certain progression.
\end{abstract}

\section{Introduction}

\subsection{Background}

Given a positive number $x$ and a base $B > 1$,  we may write $x = S_B(x) B^{k(x)}$, where $S_B(x) \in [1, B)$ is the significand and $k(x) \in \Z$. The distribution of $S_B(x)$ has interested researchers in a variety of fields for over a hundred years, as frequently it is not uniformly distributed over $[1, B)$ but exhibits a profound bias. If ${\rm Prob}(S_B(x) \le s) = \log_B(s)$ we say the system follows Benford's law, which immediately implies the probability of a first digit of $d$ is $\log_B(d+1) - \log_B(d) = \log_B(1 + 1/d)$ (at least if $d+1 \le B$); in particular, base 10 has a first digit of 1 about 30\% of the time, and 9 for only around 4.5\% of the values. This bias was first observed by Newcomb \cite{New} in the 1880s, and then rediscovered by Benford \cite{Ben} nearly 50 years later.

Many systems follow Benford's law; on the pure math side these include the Fibonacci numbers (and most solutions to linear recurrence relations) \cite{BrDu}, iterates of the 3x+1 map \cite{KonMi,LagSo}, and values of $L$-functions on the critical strip among many others; on the applied side examples range from voter and financial data \cite{Meb,Nig} to the average error in floating point calculations \cite{Knu}. See \cite{BH3, Mil} for two recent books on the subject, the latter describing many of the applications from detecting fraud in taxes, images, voting  and scientific research, \cite{BH2,Dia,Hi1,Hi2,Pin,Rai} for some classic papers espousing the theory, and \cite{BH1,Hu} for online collections of articles on the subject.

Our purpose is to explore the distribution of leading digits of components chosen from some random process. We concentrate on two related systems. The first are various $n \times n$ matrix ensembles, which of course can be viewed as vectors living in $\R^{n^2}$. The second are components of a point uniformly chosen on a unit sphere, which turn out to imply results for some of our matrix ensembles. 
 
Following the work of Montgomery \cite{Mon}, Odlyzko \cite{Od1,Od2}, Katz-Sarnak \cite{KaSa1,KaSa2}, Keating-Snaith \cite{KeSn1,KeSn2,KeSn3}, Conrey-Farmer-Keating-Rubinstein-Snaith \cite{CFKRS} and many others, random matrix ensembles in general, and the classical compact groups in particular, have been shown to successfully model a variety of number theory objects, from special values to distribution of zeros to moments. In some number theory systems Benford's law has already been observed (such as values of $L$-functions in \cite{KonMi}, or values of Fourier coefficients in \cite{ARS}); thus our work can be interpreted as providing another explanation for the prevalence of Benford's law.

We first quickly review some needed background material and then state our results.



\subsection{Haar Measure Review}

Random matrix theory has enjoyed numerous successes over the past few decades, successfully modeling a variety of systems from energy levels of heavy nuclei to zeros of $L$-functions \cite{BFMT-B, FM, Ha}. Early work in the subject considered ensembles where the matrix element were drawn independently from a fixed probability distribution $p$; this of course led to questions and conjectures on how various statistics (such as spacings between normalized eigenvalues) depended on $p$. For example, while  the density of normalized eigenvalues in matrix ensembles (Wigner's semi-circle law) was known for all ensembles where the entries were chosen independently from nice distributions, the universality of the spacings between adjacent normalized eigenvalues resisted proof until this century (see, among others, \cite{ERSY, ESY, TV1, TV2}).

Instead of choosing the matrix elements independently and having to choose a $p$, we can consider matrix groups where the Haar measure gives us a canonical choice for randomly choosing a matrix element.\footnote{These are the ensembles that turn out to be most useful in number theory, not the ones arising from a fixed distribution.} On an $n$-dimensional Lie group $G$ there exists a unique, non-trivial countably additive measure $\mu$ which is left translation invariant (so $\mu(gE) = \mu(E)$ for all $g \in G$ and $E$ a Borel set); $\mu$ is called the Haar measure. If our space is compact we may normalize $\mu$ so that it assigns a measure of 1 to $G$ and thus may be interpreted as a probability. See \cite{HR} for more details on the Haar measures and Lie groups.

We are especially interested in the case where $G\subset {\rm GL}(V)$ is a connected linear Lie group; we take $p_{i,j}$ to be the projection of $G$ onto the $i,j$-th coordinate and study the distribution of the leading digits. For many $G$ the resulting behavior is easily determined, and follows immediately from the observation that a system whose density is $\frac{1}{\log B} \frac{1}{x}$ on $[1, B)$ follows Benford's law (see definition \ref{def:diglaw}) and Theorem \ref{thm:1/xBenford}). After introducing some terminology, we state five cases which are immediately analyzed from the Haar density; Theorem \ref{diagdet1} plays a key role in our later work (Theorem \ref{SLnBenford}). These theorems are interpretations of Haar measure decompositions of classical noncompact $G$ (see \cite{HR}). Care must be taken to separate the notion of digit law for the compact and noncompact cases since many noncompact $G$ do not posses a $G$ invariant probability measure. So we have two definitions of leading digit law: for noncompact $G$, we average the measure of significands over a neighborhood of a specific one-parameter subgroup (see Definition \ref{def:diglawnoncpt} for a precise statement).  If $G$ is compact, the Haar measure affords a global average over all matrix elements. So one may think of the noncompact digit law as a local, and the compact digit law as global (see Definition \ref{def:diglawcpt}).

\begin{definition}
\label{def:diglaw}
Given a base $B\in\N$, a \emph{digit law} is a probability density function $\psi: [1,B)\to [0,1].$ A digit law satisfies a $(B,k)$ power law (for $k > 0$) if \be \psi(x)\ =\ \psi_k(x) \ := \  \frac{B^{k-1}-1}{(k-1)B^{k-1}} \frac{1}{x^k},\ee and is $B$-Benford if \be \psi(x) \ =\  \psi_1(x) \ := \  \frac{1}{\log B} \frac{1}{x}.\ee
\end{definition}


Notice that $\lim_{k\to 1} \psi_k(x) = \psi_1(x)$.

\begin{definition}
\label{def:diglawnoncpt}
Given a connected, noncompact, locally compact Lie group $G$ with Lie algebra $L(G)$, a unit direction $X\in L(G)$ which generates a one parameter subgroup $x=x(t)=\exp(tX)$ of $G$, a base $B>0$, a positive measure $\mu$ on $G$, and probability function $\psi:[1,B)\to [0,1]$ we say that $(G,d\mu,x)$ satisfies the digit law $\psi$ if the following holds: If we let
\be U_{\epsilon}(X)\ = \ \{Y+X\in L(G)\ |\ Y\perp X, |Y|<\epsilon\}\ee
be the disk of radius $\epsilon$ containing $X$ that is orthogonal to $X$ in $L(G)$, we have
\begin{equation}
\label{equ:Haardef}
{\rm Prob}(S_B(x) \le s)\ =\ \lim_{k\to\infty}\lim_{\epsilon\to 0}\frac{\sum_{l=0}^{k-1}\mu(\exp(U_{\epsilon}([\log B^l,\log B^ls)X)))}{\mu(\exp(U_{\epsilon}([0,\log B^k)X)))}=\int_0^s\psi(t)\ dt
\end{equation}
(where $k$ is a positive integer).
\end{definition}
\begin{remark}

The above definition, though somewhat involved, captures the essence of leading digit law by averaging $\mu$ in the direction of $X$ according to the significands base $B$. Since we are, in many cases, averaging the Haar measure in a specific direction, we find digit laws in components of matrix groups which are not amenable (${\rm SL}_2(\R)$, e.g.)

\end{remark}
By the Baker-Campbell-Hausdorff formula, the averaging condition \eqref{equ:Haardef} is equivalent to
\begin{equation}
\label{equ:Haardef2}
{\rm Prob}(S_B(x) \le s)\ =\ \lim_{k\to\infty}\lim_{\epsilon\to 0}\frac{\sum_{l=0}^{k-1}\mu(\exp(\log B^lX)\exp(U_{\epsilon}([0,\log s)X)))}{\mu(\exp(U_{\epsilon}([0,k\log B)X)))}.
\end{equation}

We typically take $\mu$ to be the left or right invariant Haar measure on $G$. If $\mu$ is left or bi invariant, \eqref{equ:Haardef2} becomes
\begin{equation}
\label{equ:Haardef3}
{\rm Prob}(S_B(x) \le s)\ =\ \lim_{k\to\infty}\lim_{\epsilon\to 0}\frac{k\mu(\exp(U_{\epsilon}([0,\log s)X)))}{\mu(\exp(U_{\epsilon}([0,k\log B)X)))}.
\end{equation}

\begin{thm}
\label{thm:1/xBenford}
$(\R^+, dx/x)$ is $B$-Benford.
\end{thm}
\begin{proof}
As the Lie algebra $L(R^+)=\R$ of $R^+$ is one dimensional, the perpendicular subspace to $\R$ is $\{0\}$, Thus for any $s\in [1,B)$, one has $U_{\epsilon}([0,\log s)X)=[0,\log s)X$, whence \eqref{equ:Haardef3} becomes
\be {\rm Prob}(S_B(X) \le s)\ =\ \lim_{k\to\infty}\frac{k\int_1^s dx/x}{\int_1^{B^k} dx/x}\ = \ \frac{k\log s}{k\log B}\ =\ \log_B s.\ee
\end{proof}
In the spirit of Theorem \ref{thm:1/xBenford}, when the Haar density decomposes as a product of densities on the matrix components, as it does in the next three theorems, the digit laws are easily determined from formulation \eqref{equ:Haardef3}.
\begin{thm}
\label{thm:trimatrices1}
Let $G=P$ be the group of real-valued upper triangular matrices:
\be P \ = \  \left\{
\begin{bmatrix}
a_{11} &  a_{12}  & \ldots & a_{1n}\\
0  &  a_{22} & \ldots & a_{2n}\\
\vdots & \vdots & \ddots & \vdots\\
0  &   0       &\ldots & a_{nn}
\end{bmatrix}, a_{ii}\in\R/\{0\} \right\}.\ee

The leading digit law of $A_{ii}$ for the left invariant Haar density $dg_L$ is

\begin{itemize}

\item $B$-Benford when $i=j=1$,

\item a $(B,k)$ power law when $i=j=k$, $2\leq k \leq n$,

\item uniform for $1<i<j\leq n$.

\end{itemize}

The leading digit law of $a_{ii}$ for the right invariant Haar density $dg_R$ is

\begin{itemize}

\item $B$-Benford when $i=j=n$,

\item a $(B,n-k)$ power law when $i=j=k$, $2\leq k \leq n$,

\item uniform for $1<i<j\leq n$.

\end{itemize}

\end{thm}
\begin{proof}

The left invariant Haar measure on $P$ has density
\be dg_L \ = \  \frac{1}{a_{11}a_{22}^2\cdots a_{nn}^n}\ \prod_{i<j}\ da_{ij}\ee
and the right invariant Haar measure on $P$ has density
\be dg_R \ = \  \frac{1}{a_{11}^na_{22}^{n-1}\cdots a_{nn}}  \prod_{i<j}\ da_{ij},\ee
where $da_{ij}$ is the Lebesgue density on $\R$ in both cases. All leading digit laws follow.

\end{proof}

\begin{thm}
Let $D$ be the group of real-valued diagonal matrices:
\be D\ = \  \left\{
\begin{bmatrix}
a_{11} &   \ldots & 0\\
\vdots & \ddots & \vdots\\
0        &\ldots & a_{nn}
\end{bmatrix}, a_{ii}\in\R/\{0\} \right\}.\ee
For each $i$ between $1$ and $n$, the leading digit law of $a_{ii}$ with respect to the bi-invariant Haar density $dg$ is $B$-Benford.
\end{thm}
\begin{proof}
The bi-invariant Haar measure on $D$ is
\be dg\ = \  \frac{1}{a_{11}a_{22}\cdots a_{nn}}\ da_{11}da_{22}\cdots da_{nn},\ee
where $da_{ii}$ is the Lebesgue measure on $\R$. The digit laws follow.
\end{proof}

\begin{thm}
\label{diagdet1}
Let $D_1$ be the group of real-valued, determinant $1$ diagonal matrices:
\be D_1 \ = \  \left\{
\begin{bmatrix}
a_{11} &   \ldots & 0\\
\vdots & \ddots & \vdots\\
0        &\ldots & a_{nn}
\end{bmatrix}, \prod{a_{ii}}\ = \ 1 \right\};\ee
For each $i$ between $1$ and $n$, the leading digit law of $a_{ii}$ with respect to the bi-invariant Haar density $dg$ is $B$-Benford.
\end{thm}
\begin{proof}
$D_1$ is diffeomorphic to the graph of \be (a_{11},\ldots, a_{n-1,n-1})\ \mapsto\ \frac{1}{a_{11}a_{22}\cdots a_{n-1,n-1}}\ee
and hence is diffeomorphic to an open sub-manifold of $\R^{n-1}$. The bi-invariant Haar measure on $D_1$ is thus
\be dg\ = \  \frac{1}{a_{11}a_{22}\cdots a_{n-1,n-1}}da_{11}da_{22}\cdots da_{nn},\ee
where $da_{ii}$ is the Lebesgue measure on $\R$. The digit laws follow.
\end{proof}

\subsection{Main Results}

Our first result concerns the distribution of entries from ${\rm SL}_n(\R)$. Denote by $L,U,D_1\subset G$ the subgroups of unipotent lower triangular, unipotent upper triangular, and diagonal subgroup of ${\rm SL}_n(\R)$. Then $g\in G$ can be uniquely expressed as $g=lud, l\in L,u\in U, d\in D_1$. Note that each of $L,U,D_1$ is topologically closed in ${\rm SL}_n(\R)$, and hence each is a Lie subgroup of $G$. If $\fl,\fu,\fd_1$ be the Lie algebras of $L,U,D$ respectively then $\fl,\fu,\fd_1$ have the vector space basis (which we review in Appendix \ref{app:llg}):
\be \fl \ = \  \spans_{\ \!\! \R}(\{ {E_{i,j}}\}_{i>j}),\ \ \ \fu \ = \  \spans_{\ \!\! \R}(\{ {E_{i,j}}\}_{i<j}), \ \ \ \fd_1\ = \ \spans_{\ \!\! \R}(E_{i,i}-E_{i+1,i+1})_{1\leq i \leq n-1},\ee where $E_{i,j}$ is the $n\times n$ matrix with $1$ in the $(i,j)$ position and zeroes elsewhere.

\begin{thm}\label{SLnBenford}
Let $dg$ be the normalized Haar measure on ${\rm SL}_n(\R)$, $\phi\in C_c(G)$. Then
\be \int_G \phi(g)dg\ = \ \int_{\fl}\int_{\fu}\int_D\phi(\exp(X)\exp(Y)a)\; \! da\; \! dX\; \! dY,\ee
where $dX,dY$ are the Lebesgue measures on $\fl,\fu$ and
\be da\ = \ \prod_{i=1}^{n-1} \frac{da_{ii}}{a_{ii}} \ee
is the Haar measure on $D_1$. Consequently, the joint distribution of diagonal components is a product of $B$-Benford measures.
\end{thm}
The next corollary follows immediately from the invariance of $dg$ on $SL_n(\R)$:
\begin{cor}
\label{cor:othercomponentsSLn}
Let $P,Q\in {\rm SL}_n(\R)$ be even order permutation matrices. For $A\in {\rm SL}_n(\R)$, the joint distribution of the diagonal components of $PAQ$ are a product of $B$-Benford measures.
\end{cor}
In other words, the joint distribution of $n$ components is a product of $B$-benford measures if there is an even permutation of the rows and columns which sends the $n$ components to the diagonal components. As an immediate consequence of the above, we obtain results on the behavior of determinants of matrices from ${\rm GL}_n(\R)^+$ (Theorem \ref{thm:glndetBenford}). For other results related to Benford's law and matrices, see \cite{B--}, who prove that as the size of matrices with entries i.i.d.r.v. from a nice fixed distribution tends to infinity, the leading digits of the $n!$ terms in the determinant expansion converges to Benford's law. Also see \cite{BH3} for results arising from powers of fixed matrices.

When $G$ is compact, the Haar measure may be normalized to be an invariant probability measure on $G$, affording a global definition of digit law, stated next.
\begin{definition}
\label{def:diglawcpt}
Fix a base $B>0$. Let $G$ be a compact connected Lie group, $\mu$ a positive countably additive probability measure on $G$, $f:G\to \R$ measurable. We saw that $(G,\mu,f)$ satisfies the digit law $\psi$ if
\begin{equation}
{\rm Prob}(S_B(f(g)) < s)\ =\ \int_{1}^s\psi(x)\ dx.
\end{equation}
\end{definition}
We shall see that when $G=O(n)$ or $U(n)$, $f$ is a projection of $G$ onto the $(i,j)$-th component and $\mu$ is Haar, the digit laws come as a consequence of digit laws from a point drawn at random from a unit sphere (see Corollary \ref{cor: cptgrpdiglaw}). So our next result yields digit laws for components of a point drawn at random on an $n$-dimensional sphere of radius $r$: \be S^n(r) \ := \  \{x\in\R^{n+1}: \vert x\vert = r\}.\ee

\ \\
\texttt{We adopt the notational convention for the unit sphere: $S^n:=S^n(1)$.}\ \\

\begin{thm}\label{thm:firstcomponentsphere} Let $x_1$ be the first component of an $x\in S^n$ chosen uniformly at random. We have for $1 \le a \le b \le B$ that
\begin{equation}\label{eq:probdensityx1sphere}
{\rm Prob}(a< S_B(x_1)<b) \ = \ \frac{2}{\sqrt{\pi}}\frac{\Gamma(n/2+1/2)}{\Gamma(n/2)}\sum_{i=1}^\infty \int_{a B^{-i}}^{b B^{-i}}(1-x_1^2)^{n/2-1}\ dx_1.
\end{equation}
\end{thm}

As $n\to\infty$, Stirling's formula implies the above converges to integrating a Gaussian density, where ${\rm erf}$ is the standard error function: \be {\rm erf}(x) \ := \ \frac{2}{\sqrt{\pi}} \int_0^x e^{-t^2}dt. \ee

\begin{lem}\label{WTerf}
Fix a base $B>1$ and $1\leq a<b < B$. Let $x_1$ and $x$ be as in Theorem \ref{thm:firstcomponentsphere}. As $n \to \infty$, ${\rm Prob}(a\leq S_B(x_1) < b)$ is well-approximated by
\begin{eqnarray} \sum_{i=1}^{\infty} \frac{2}{\sqrt{\pi}}\int_{\sqrt{\frac{n}{2}}\frac{a}{B^i}}^{\sqrt{\frac{n}{2}}\frac{b}{B^i}}e^{-x^2}\ dx \ = \  \sum_{i=1}^{\infty} \left( \erf\left(\sqrt{\frac{n}{2}}\frac{b}{B^i}\right) - \erf\left(\sqrt{\frac{n}{2}}\frac{a}{B^i}\right)\right),\end{eqnarray}
in the sense that
\be \lim_{n\to \infty} \left| {\rm Prob}(a\leq S_B(x_1) < b) - \sum_{i=1}^{\infty} \frac{2}{\sqrt{\pi}}\int_{\sqrt{\frac{n}{2}}\frac{a}{B^i}}^{\sqrt{\frac{n}{2}}\frac{b}{B^i}}e^{-x^2}\ dx\ \right|\ =\ 0. \ee
\end{lem}

\begin{rek}\label{rek:distrsphereslimitperiodic}
Lemma \ref{WTerf} has an interesting consequence.  First, consider the sequence of spheres $S^{nB^{2\ell}}$, $\ell\in\N$. For $n$ sufficiently large, with $\sqrt{\frac{n}{2}}\frac{1}{B}>4$, Then
\begin{eqnarray}
{\rm Prob}(a\leq S_B(x_1) < b, x\in S^{nB^{2\ell}}) & \ \approx\ \ & \sum_{i=1}^{\infty} \erf\left(\sqrt{\frac{nB^{2\ell}}{2}}\frac{b}{B^i}\right) - \erf\left(\sqrt{\frac{nB^{2\ell}}{2}}\frac{a}{B^i}\right)\nonumber\\
& \ = \ &   \sum_{i=1-\ell}^{\infty} \erf\left(\sqrt{\frac{n}{2}}\frac{b}{B^i}\right) - \erf\left(\sqrt{\frac{n}{2}}\frac{a}{B^i}\right).
\end{eqnarray}

By choice of $n$, the additional terms from extending the sums to all $i$ are more than $4$ standard deviations from the mean, and contribute negligibly to the sum in the limit. Hence for $n$ sufficiently large,
\be \lim_{\ell\to\infty} {\rm Prob}(a\leq S_B(x_1) < b, x\in S^{nB^{2\ell}}) \ = \  \sum_{i=-\infty}^{\infty} \erf\left(\sqrt{\frac{n}{2}}\frac{b}{B^i}\right) - \erf\left(\sqrt{\frac{n}{2}}\frac{a}{B^i}\right).\ee

For fixed $n\in\N$, it follows that the leading digit law of $x_1$ in $S^{nB^{2\ell}}$, as $\ell \to \infty$, tends to the digit law $F_n: [1,B)\to [0,1)$ whose cumulative distribution function is given by
\begin{equation}
\label{equ:limitingspherelaws}
F_n(x) \ := \  \sum_{i=-\infty}^{\infty} \erf\left(\sqrt{\frac{n}{2}}\frac{x}{B^i}\right) - \sum_{i=-\infty}^{\infty} \erf\left(\sqrt{\frac{n}{2}}\frac{1}{B^i}\right).\ee  As $F_n(x)= F_{nB^{2}}$ for any $n\in\N$, it follows that leading digit law of $x_1$ in $S^k$, $k\to \infty$, falls into the periodic cycle of $B^2-1$ limiting digit laws $F_n$, $1\leq n <B^2$ as defined in \eqref{equ:limitingspherelaws}. We plot a representative set of $n$ in Figure \ref{fig:spherebenford}. \end{rek} Lemma \ref{WTerf} and its consequences can be generalized to a fixed number of components; we do this in Lemma \ref{WTerf2}.

\begin{figure}
\begin{center}
\scalebox{.56}{\includegraphics{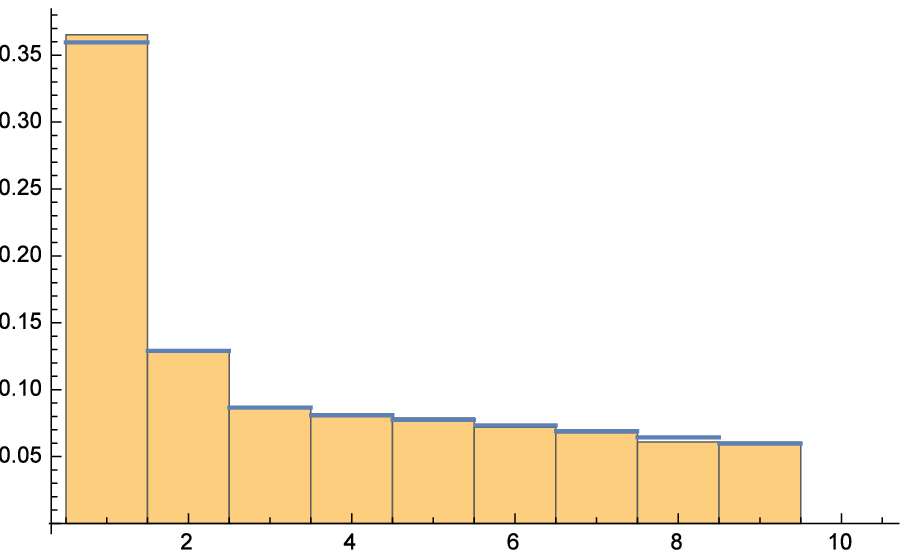}} \ \ \scalebox{.56}{\includegraphics{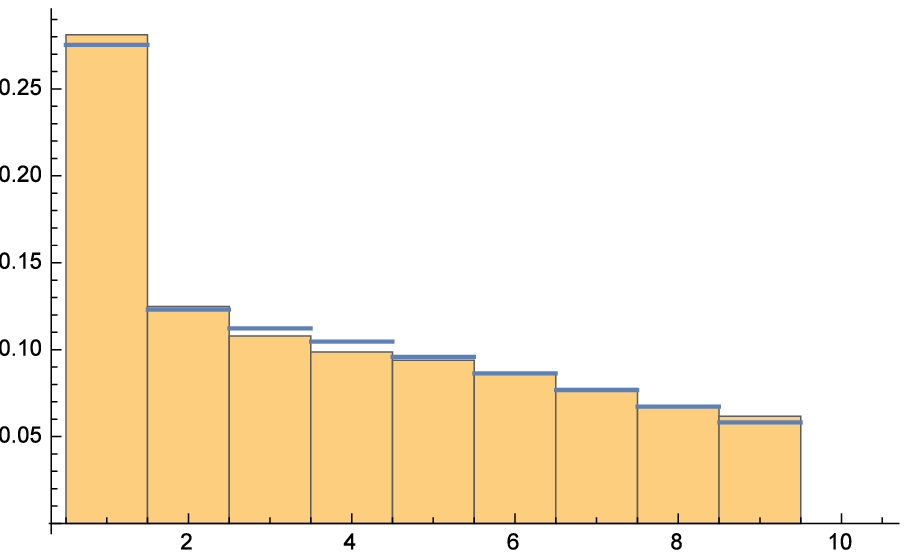}} \ \
\scalebox{.56}{\includegraphics{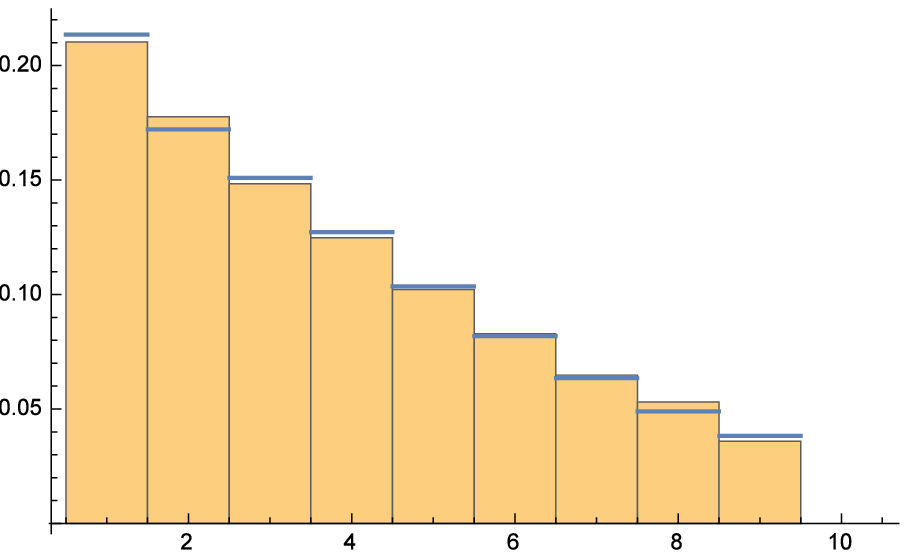}}\\
\scalebox{.56}{\includegraphics{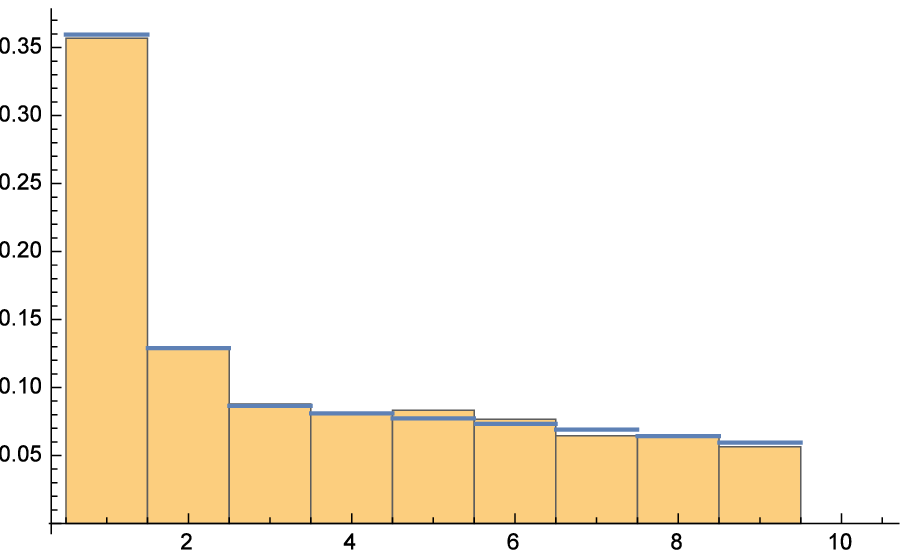}} \ \ \scalebox{.56}{\includegraphics{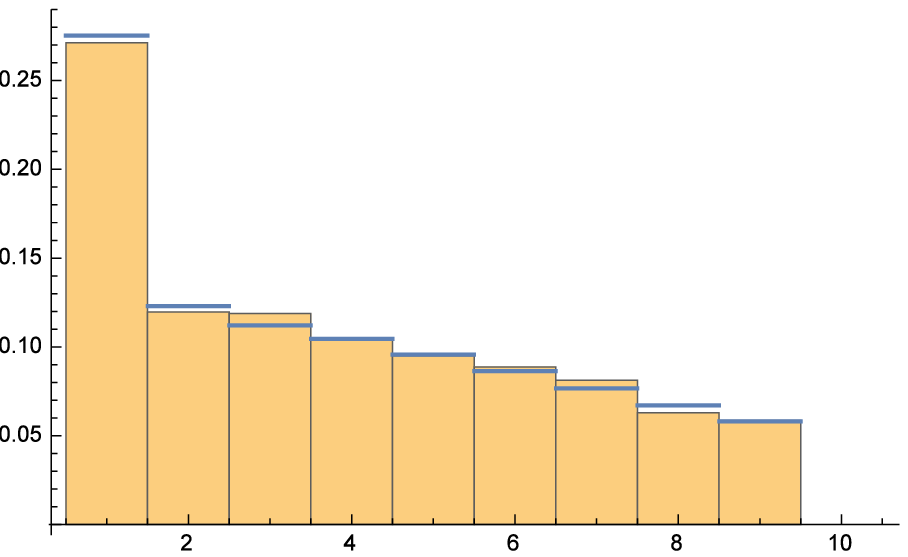}} \ \
\scalebox{.56}{\includegraphics{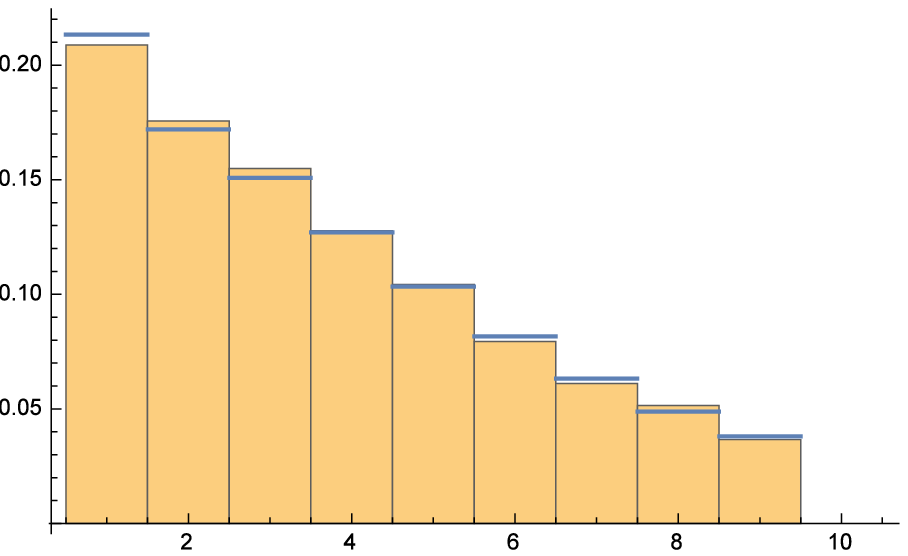}}\\
\caption{\label{fig:spherebenford} The distribution of the first digits base $B=10$ of the first component of points uniformly chosen on a sphere with $n$ components. Top row: $n \in \{100, 200, 500\}$. Bottom row: $n \in \{10000, 20000, 50000\}$. Notice the periodicity when $n$ increases by a factor of $B^2 = 100$.}
\end{center}\end{figure}
The spherical digit law in one component (Lemma \ref{WTerf}) yield digit laws for the compact matrix group $O_n(\R)$, stated next.
\begin{cor}\label{cor: cptgrpdiglaw}
The leading digit law in the $(i,j)$ component of ${\rm O}_n(\R)$ with respect to Haar is the leading digit law of $x_1$ in $S^{n-1}$ with respect to the uniform measure.
\end{cor}

\begin{proof}
As $O_n(\R)$ contains every permutation matrix $P\in {\rm GL}_n(\R)$, there exist permutation matrices $P,Q\in {\rm GL_n(\R)}$ such that $PAQ\in O_n(\R)$ sends the $(i,j)$ entry to the $(1,1)$ entry. By invaraiance of $dg$, it suffices to prove the Corollary for the $(1,1)$ component of $O_n(\R)$. Recall that any matrix $A\in {\rm O}_n(\R)$ satisfies $A^TA=I$, so the columns of ${\rm O}_n(\R)$ form an orthonormal basis of $\R^n$. We may therefore embed ${\rm O}_n(\R)$ in the product of $n$ spheres $S^{n-1}\times \cdots \times S^{n-1}$. Consider the construction of a matrix in ${\rm O}_n(\R)$ one column at a time from left to right. The first column $c_1$ can be selected arbitrarily from $S^{n-1}$. The second column $c_2$ is a vector selected in the orthogonal plane to $c_1$ in $S^{n-1}$, a set which is isometric to $S^{n-2}$. In general the $i$th column is selected in the orthogonal hyperplane to $c_1,\ldots,c_{i-1}$ in $S^{n-1}$, which is a set isometric to $S^{n-{i}}$  Since the ${\rm O}_n(\R)$ action on a subset $A\subset {\rm O}_n(\R)$ preserves the Haar measure of $A$, there is a measure preserving transformation between a basis for the Haar measurable sets of ${\rm O}_n(\R)$ and measurable subsets $A_1\times A_2,\ldots,A_n\subset S^{n-1}\times S^{n-2}\times S^0$ equipped with the uniform measure on $S^i$. Therefore, the digit law of the $(1,1)$ component of ${\rm O}_n(\R)$ is equal to the digit law of $S^{n-1}$ with the uniform measure. The leading digit law follows.
\end{proof}

Thus one sees the same asymptotic periodicity in the leading digit laws in the $(i,j)$ component of ${\rm O}_n(\R)$ with period $B^2-1$ in $n$. By invariance of $dg$, it follows that Lemma \ref{cor: cptgrpdiglaw}, \ref{WTerf2}, and formulas therein yield leading digit laws for a fixed number of components of ${\rm O}_n(\R)$, so long as all components lie in the same row or column. Lastly, analogous digit laws for the real an imaginary parts in a fixed number of components ${\rm U}_n(\C)$ are immediate, since ${\rm U}_n(\C)$ contains every permutation matrix and the first column of ${\rm U}_n(\C)$ is a point on $S^{2n-1}$.

\begin{remark}
We leave the leading digit laws of a hyperbola as future research.
\end{remark}

\ \\

We prove Theorem \ref{SLnBenford} on components of ${\rm SL}_n(\R)$ in \S\ref{sec:proofSLnBenford} (see also Appendix \ref{sec:appgeomcone} for a more geometric proof in two dimensions), and then Theorem \ref{thm:firstcomponentsphere} in \S\ref{sec:thmfirstcompsphere}, discussing some additional consequences (we have already shown above how it yields digit laws for the classical compact groups). We then finish with some concluding remarks and thoughts on future research.





\section{Proof of Theorem \ref{SLnBenford}}\label{sec:proofSLnBenford}

Let $L,U,D_1$ be lower, upper, and diagonal matrices determinant $1$ matrices, $\fl,\fu,\fd_1$ be as before; we can calculate the density of $dg$ with respect to the decomposition $G=LUD_1$. Pick any $g_0$ in $G$, and parametrize $g$ in a neighborhood of $g_0$ using exponential coordinates
\be g\ = \ g(X,Y,Z)\ = \ g_0\exp X\exp Y\exp Z.\ee
If we let
\be g(t) \ = \  g(tX,Y,Z)\ = \ g_0\exp tX\exp Y\exp Z\ee
where $X\in\fl,Y\in\fu,Z\in\fd_1$, then $\fl+\fu+\fd_1=\fg$. It follows that the derivative at $g_0$ in the direction of $X$ is
\be g'(t)\ = \ g_0(\exp tX)X\exp Y\exp Z,\ee
so that
\begin{align}
g(t)^{-1}g'(t) &\ = \ (g_0\exp tX\exp Y\exp Z)^{-1}g_0(\exp tX)X\exp Y\exp Z\nonumber\\
&\ = \ \Ad(\exp Z\exp Y)^{-1}(X)\ = \ e^{-\ad Y}e^{-\ad Z}X.
\end{align}
By a change of variables and left invariance, the differential with respect to coordinate bases of $\fl,\fu,\fd$ is given by the block matrix
\be \left[\begin{array}{c | c | c}
[\Ad(\exp Z\exp Y)^{-1}(X)]_{\fl} & 0 & 0 \\ \hline
* & [\Ad(\exp Z)^{-1}(Y)]_{\fu} & 0 \\ \hline
* & * &  Z_{\fd_1}\end{array}\right],\ee
where $[\Ad(\exp Z\exp Y)^{-1}(X)]_{\fl}$ is the part of $\Ad(\exp Z\exp Y)^{-1}(X)$ that lies in the subspace $\fl$. Thus, the volume element on $G=LUD_1$ in a neighborhood of $g_0$ is given by
\be |\det \Ad((ud)^{-1})_{\fl}||\det \Ad((d)^{-1})_{\fu}|\ee
and is independent of $g_0$. By Fubini's theorem, $\int_G \phi(g)dg$ is
\be \int_{\fl}\int_{\fu}\int_{\fd_1}\phi(\exp(X)\exp(Y)\exp(Z)) |\det \Ad((ud)^{-1})_{\fl}||\det \Ad((d)^{-1})_{\fu}|\ \: dZ dX dY.\ee
Using the ordinary basis $\{E_{i,j}\}_{i<j}$ of $\fu$, the adjoint action of the diagonal subgroup $D$ on $\fu$ is
\be \Ad(d^{-1})E_{i,j}\ = \  d^{-1}E_{i,j}d \ = \ \frac{d_{jj}}{d_{ii}}E_{i,j}.\ee
It follows that (with respect to exponential coordinates of the first kind)
\be \det \Ad(d^{-1})_{u} \ = \  \prod_{1\leq i<j\leq n} \frac{d_{jj}}{d_{ii}}.\ee
With respect to the basis $\{E_{i,j}\}_{i>j}$ of $\fl$, one can see that
\be \Ad((ud)^{-1})E_{i,j} \ = \  (ud)^{-1}E_{i,j}ud \ee takes the block form
\be \left[\begin{array}{c | c }
0_{i,n-j} & * \\ \hline
0_{n-i,n-j} & 0_{n-i,j}
\end{array}\right].\ee

Ordering the basis of $\fl$ along sub-diagonals, $\Ad((ud)^{-1})_{\fl}$ becomes upper triangular, with \be \Ad(u^{-1})_{\fl}\ = \  {\rm id}_{\fl}\ee for all $u\in R$. Therefore, $|\det \Ad((ud)^{-1})_{\fl}| =  |\det \Ad(d^{-1})_{\fl}|$, and the adjoint action of $D$ on $\fl$ is simply \be \Ad(d^{-1})_{\fl}E_{i,j} \ = \  d^{-1}E_{i,j}d \ = \  (d^TE_{i,j}^T(d^{-1})^{T})^T\ = \  (\Ad(d)E_{j,i})^T\ = \  \frac{d_{ii}}{d_{jj}}E_{i,j}.\ee
Therefore \be |\det\Ad((ud)^{-1})_{\fl}||\det \Ad((d)^{-1})_{\fu}|\ = \ 1\ee and Theorem \ref{diagdet1} completes the proof. \hfill \qed

We provide another proof of Theorem \ref{SLnBenford} through a geometric approach, based on the area of the hyperbolic sector, in Appendix \ref{sec:appgeomcone}.


\section{Proof and Consequences of Theorem \ref{thm:firstcomponentsphere}}\label{sec:thmfirstcompsphere}

For $r>0$, let \be S^n(r) \ = \  \{x\in\R^{n+1}\mid \vert x\vert = r\}\ee be the sphere of radius $r$ in $\R^{n+1}$. Denote by $V_n(r)$ and $S_n(r)$ the volume and surface area of $S^n(r)$ (recall we write $S^n$ for the unit sphere). Fix a base $B>1$ and let $S_B(x)$ be the significand function, i.e., $S_B(|y|)\in [1,B)$ is the unique number satisfying \be |y| \ = \  S_B(|y|)B^k\ee for some $k\in \Z$.

\begin{proof}[Proof of Theorem \ref{thm:firstcomponentsphere}]
Pick a point $x\in S^n$ uniformly at random, and let $x_1$ be the first component of $x$. We are interested in the leading digit distribution of $x_1$. By symmetry, the distribution for other components will be similar. Notice in $\R^{n+1}$ that for $0<a<1$ \be \{x_1 = a\} \cap S^n \ = \  S^{n-1}(\sqrt{1-a^2}).\ee Approximating the surface area in the strip $\{a<x_1<b,x\in S^n\}$ by a frustum, it follows for $n>0$  that
\begin{equation}\label {SAsphere}
{\rm Prob}(a<x_1<b,x\in S^n) \ = \  \frac{\int_a^b\frac{1}{\sqrt{1-x_1^2}}S_{n-1}(\sqrt{1-x_1^2}) dx_1}{S_n(1)}.
\end{equation}

By the familiar relationships $S_n(r) = V_n^{'}(r) = \frac{n+1}{r}V_{n}(r)$, and the closed form solution \be V_n(r) \ = \  \frac{\pi^{(n+1)/2}r^{n+1}}{\Gamma(\frac{n+1}{2}+1)},\ee
we find
\begin{align}
{\rm Prob}(a<x_1<b,x\in S^n) &\ = \  \frac{\int_a^b\frac{1}{\sqrt{1-x_1^2}}S_{n-1}(\sqrt{1-x_1^2})\ dx_1}{S_n(1)}\nonumber\\
			      &\ = \  \frac{n\int_a^b\frac{1}{1-x_1^2}V_{n-1}(\sqrt{1-x_1^2})\ dx_1}{(n+1) V_n(1)}\nonumber\\
			     &\ = \  \frac{1}{\sqrt{\pi}}\frac{n\Gamma(n/2+3/2)}{(n+1)\Gamma(n/2+1)}\int_a^b(1-x_1^2)^{n/2-1}\ dx_1\nonumber\\
			     &\ = \  \frac{1}{\sqrt{\pi}}\frac{\Gamma(n/2+1/2)}{\Gamma(n/2)}\int_a^b(1-x_1^2)^{n/2-1}\ dx_1.
\end{align}
Now, fix $a,b$, where $1\leq a < b \leq B.$ By symmetry, we may double the digit distribution in the positive half-space $x_1 > 0$. Thus
\begin{align}\label{eq:whatwaseq1}
{\rm Prob}(a<S_B(x_1)<b,x\in S^n) \ = \ \frac{2}{\sqrt{\pi}}\frac{\Gamma(n/2+1/2)}{\Gamma(n/2)}\sum_{i=1}^\infty \int_{a*B^{-i}}^{b*B^{-i}}(1-x_1^2)^{n/2-1}\ dx_1.
\end{align}
\end{proof}

For example, when $n=1$ we have
\begin{align}
{\rm Prob}(a<x_1<b,x\in S^1) &\ = \  \frac{\int_a^b\frac{1}{\sqrt{1-x_1^2}}S_{0}(\sqrt{1-x_1^2}) dx_1}{S_1(1)}\nonumber\\
&\ = \ \frac{1}{\pi}\int_a^b\frac{1}{\sqrt{1-x_1^2}}\ dx_1\nonumber\\				
&\ = \  \frac{\arcsin(b)-\arcsin(a)}{\pi},
\end{align}
and thus
\begin{align}
{\rm Prob}(a<S_B(x_1)<b,x\in S^1) &\ = \  \frac{2}{\pi} \sum_{i=1}^\infty \left(\arcsin\left(\frac{b}{B^i}\right)-\arcsin\left(\frac{a}{B^i}\right)\right),
\end{align} while for $n=2$
\begin{align}
{\rm Prob}(a\leq x_1<b,x\in S^2) &\ = \  \frac{\int_a^b\frac{1}{\sqrt{1-x_1^2}}S_{1}(\sqrt{1-x_1^2}) dx_1}{S_2(1)}\nonumber\\
				&\ = \  \frac{\int_a^b2\pi dx_1}{4\pi }\nonumber\\
				&\ = \  \frac{b-a}{2},
\end{align}
which implies \begin{align}
{\rm Prob}(a\leq S_B(x_1)<b,x\in S^2) &\ = \  \frac{b-a}{B-1}.
\end{align}

The leading digit distribution on $S^2$ is uniform, with respect to any base, which is akin to the fact that equal width slices of a spherical loaf contain the same amount of crust. Our main theorem is an asymptotic result. In a sense, the digit distribution is found by applying Stirling's formula and integrating the standard Gaussian.

\begin{proof}[Proof of Lemma \ref{WTerf}]
Let $a,b\in \R$ satisfy $1\leq a<b < B$. Recall, from the above derivation (see \eqref{eq:whatwaseq1}) that ${\rm Prob}(a<S_B(x_1)<b,x\in S^n)$ equals \be \frac{2}{\sqrt{\pi}}\frac{\Gamma(n/2+1/2)}{\Gamma(n/2)}\sum_{i=1}^\infty \int_{a*B^{-i}}^{b*B^{-i}}(1-x^2)^{n/2-1}\ dx.\ee
By Stirling's approximation \be \frac{\Gamma(n/2+1/2)}{\Gamma(n/2)} \ = \  \sqrt{\frac{n}{2}}+O(1).\ee Using this with the substitution $x=y\sqrt{2/n}$, $dx = dy\sqrt{2/n}$ in the integrand yields, for $n$ sufficiently large and $x \in S^n$, that
\begin{eqnarray}
{\rm Prob}(a\leq S_B(x_1) < b) & \ = \  &\left(\sqrt{\frac{n}{2}}+O(1)\right)\frac{2}{\sqrt{\pi}}\sum_{i=1}^\infty \int_{\sqrt{\frac{n}{2}}\frac{a}{B^i}}^{\sqrt{\frac{n}{2}}\frac{b}{B^i}}\left(1-\frac{y^2}{n/2}\right)^{n/2-1}\sqrt{\frac{2}{n}}\ dy\nonumber\\
& \ = \  &\left(1+O\left(\frac{1}{\sqrt{n}}\right)\right)\frac{2}{\sqrt{\pi}}\sum_{i=1}^\infty \int_{\sqrt{\frac{n}{2}}\frac{a}{B^i}}^{\sqrt{\frac{n}{2}}\frac{b}{B^i}}\left(1-\frac{y^2}{n/2}\right)^{n/2-1} dy\nonumber\\
& \approx &\left(1+O\left(\frac{1}{\sqrt{n}}\right)\right)\frac{2}{\sqrt{\pi}}\sum_{i=1}^\infty \int_{\sqrt{\frac{n}{2}}\frac{a}{B^i}}^{\sqrt{\frac{n}{2}}\frac{b}{B^i}}e^{-y^2}\ dy.
\end{eqnarray}
\end{proof}

Lemma \ref{WTerf} can be generalized to many components. Pick a point at random on the unit sphere $S^n\subset\R^{n+1}$, and consider the first $k$ components $x_1,\ldots,x_k$ $(k<n+1)$. We are interested in the joint distribution of leading digits that appear in the first $k$ components. Similar to the analysis above, for a point $(a_1,a_2,\ldots,a_k)$ in the open unit disk $D^{k}$, notice that the other $n-k+1$ components lie in a $n-k$ sphere of radius $\sqrt{1-a_1^2-\cdots-a_k^2}$. Exploiting the rotational symmetry in the last $n-k+1$ components, we may parametrize the surface element $dS_n$ of $S^n$ by $D^{k}$ as
\begin{eqnarray} dS_n(x_1,\ldots,x_k)  & \ = \ & S_{n-k}(\sqrt{1-x_1^2-\cdots -x_k^2})\ dS_k(x_1,\ldots,x_k)\nonumber\\
dS_n(x_1,\ldots,x_k)  & \ = \ & S_{n-k}(\sqrt{1-x_1^2-\cdots -x_k^2})\frac{1}{\sqrt{1-x_1^2-\cdots -x_k^2}}\ dx_1\cdots dx_k.\ \end{eqnarray}

\begin{lem}
\label{WTerf2}
Fix an integer $k>0$, and let $a_1,b_1,\ldots,a_k,b_k \in \R$ satisfy $0\leq |a_i|<|b_i| < 1$, $1\leq i\leq k$. For $n>k$ sufficiently large,
\be \label{equ:mulicompspheres} {\rm Prob}(|a_1|\leq |x_1| < |b_1|,\ldots,|a_k|\leq |x_k| < |b_k|, x\in S^n)\ee
is well-approximated by
\be \label{equ:mulicompgaussian}   \prod_{i=1}^k\frac{2}{\sqrt{\pi}}\int_{a_i\sqrt{\frac{n}{2}}}^{b_i\sqrt{\frac{n}{2}}}e^{-x^2}\ dx\ee
in the sense that the difference between \eqref{equ:mulicompspheres} and \eqref{equ:mulicompgaussian} tends to zero as $n\to\infty$.
\end{lem}

\begin{proof}
Similar to Lemma \ref{WTerf}, we need only worry about when $a_i,b_i>0$. By symmetry and substitution
\begin{align}
	&{\rm Prob}(a_1<x_1<b_1,\ldots,a_k<x_k<b_k,x\in S^n)\nonumber\\
	&\ = \ \displaystyle2^k\int_{a_1<x_1<b_1,\ldots,a_k<x_k<b_k}dS(x_1,\ldots,x_k)\nonumber\\
	 &\ = \  \frac{2^k}{S_n(1)}\int_{a_1}^{b_1}\cdots\int_{a_k}^{b_k}S_{n-k}(\sqrt{1-x_1^2-\cdots -x_k^2})\frac{1}{\sqrt{1-x_1^2-\cdots -x_k^2}}\ dx_1\cdots dx_k\nonumber\\
	&\ = \ \frac{2^k(n-k+1)}{(n+1)V_n(1)}\int_{a_1}^{b_1}\cdots\int_{a_k}^{b_k}V_{n-k}(\sqrt{1-x_1^2-\cdots -x_k^2})\frac{1}{\sqrt{1-x_1^2-\cdots -x_k^2}}\ dx_1\cdots dx_k\nonumber\\
	&\ = \ \left(\frac{2}{\sqrt\pi}\right)^k\frac{\Gamma(n/2+1/2)}{\Gamma(n/2-k/2+1/2)}\int_{a_1}^{b_1}\cdots\int_{a_k}^{b_k}(1-x_1^2-\cdots -x_k^2)^{(n-k-1)/2}\ dx_1\cdots dx_k.
\end{align}
Stirling's approximation \be \frac{\Gamma(n/2+1/2)}{\Gamma(n/2-k/2+1/2)} \ = \  \left(\frac{n}{2}\right)^{k/2}\!\!\! +\ O(n^{k/2-\epsilon})\ee and the substitutions $x_i = y_i / \sqrt{n/2}$ complete the proof.
\end{proof}

\begin{cor}
Fix an integer $k>0$. For any base $B>1$, and $a_1,b_1,\ldots,a_k,b_k \in \R$ satisfying $1\leq a_i< b_i < B$, $1\leq i\leq k$, we have for $n>k$ sufficiently large
\be \label{equ:mulicompspheresdig} {\rm Prob}(a_1\leq S_B(x_1) < b_1,\ldots, a_k \leq S_B(x_k) < b_k, x\in S^n)\ee
is well-approximated by
\be \label{equ:mulicompgaussiandig} \prod_{j=1}^k\sum_{i=1}^{\infty} \frac{2}{\sqrt{\pi}}\displaystyle\int_{\sqrt{\frac{n}{2}}\frac{a_j}{B^i}}^{\sqrt{\frac{n}{2}}\frac{b_j}{B^i}}e^{-x^2}\ dx\ee
in the sense that the difference between \eqref{equ:mulicompspheresdig} and \eqref{equ:mulicompgaussiandig} tends to zero as $n\to\infty$.
In particular, the joint leading digit distribution of the first $k$ components is asymptotically periodic in $n$, with period $B^2$, tending to one of the $B^2-1$ limiting distributions \be \prod_{j=1}^kF_n(x_j)\ = \ \prod_{j=1}^k\sum_{i=-\infty}^{\infty}\left( \erf\left(\sqrt{\frac{n}{2}}\frac{x_j}{B^i}\right)- \erf\left(\sqrt{\frac{n}{2}}\frac{1}{B^i}\right)\right)\ \ 1\leq n < B^2.\ee
\end{cor}

\section{Conclusions and Future Work}\label{sec:conclusion}

Our results above can serve as a means for detecting underlying symmetries of a physical system. For example, imagine we are trying to construct matrices from one of the classical compact groups according to Haar measure (see \cite{Mez} for a description of how to do this). We can use our digit laws as a test of whether or not we are simulating the matrices correctly. It would be interesting to generalize the arguments above to other groups of matrices, including those over fields other than the reals.

\appendix

\section{Linear Lie groups}\label{app:llg}

A Lie group $G\subset {\rm GL}(V)$ is a group equipped with a differentiable structure such that the binary operation $G\times G\to G$ is differentiable. The Lie algebra $L(G)$ may be naturally identified with the tangent space $T_e(G)$ to the identity. For a direction $X\in L(G)$ there is a unique one parameter subgroup $\exp(tX)$, $t\in\R$, in the direction of $X$ and the map $\exp: L(G)\to G$ is a local diffeomorphism. Let $E_{ij}$ be the $n\times n$ matrix with $1$ in the $(i,j)$ entry and zeroes elsewhere.

The groups in this paper are the following.

\begin{itemize}
\item The general linear group ${\rm GL}_n(\R)$ of matrices of nonzero determinant and its Lie algebra $\mathfrak{gl}_n(\R)$ of all $n\times n$ matrices.

\item The special linear group: ${\rm SL}_n(\R) = \{A\in {\rm GL}_n(V)\mid \det A =1\}$ and its Lie algebra $\mathfrak{sl}_n(\R) = \{X\in gl_n(\R)\mid {\rm tr} X =0\}$ of traceless matrices.

\item The space of diagonal matrices ${\rm D}\subset GL_n(\R)$ with nonzero diagonal entries and its Lie algebra $\fd$ of diagonal matrices with entries in $\R$.

\item The space of diagonal matrices determinant ${\rm D}_1(\R)\subset GL_n(\R)$ with nonzero diagonal entries and its Lie algebra $\fd_1$ of traceless matrices with entries in $\R$.

\item The space of upper triangular matrices ${\rm U}(\R)\subset GL_n(\R)$ with nonzero diagonal entries and its Lie algebra $\fu$ of upper triangular matrices with entries in $\R$.

\item The space of lower triangular matrices ${\rm L}(\R)$ with nonzero entries and its Lie algebra $\fl$ of lower triangular matrices with entries in $\R$.
\item Note that
\be \fl \ = \  \spans_{\ \!\! \R}(\{ {E_{i,j}}\}_{i>j}),\ \ \ \fu \ = \  \spans_{\ \!\! \R}(\{ {E_{i,j}}\}_{i<j}), \ \ \ \fd_1\ = \ \spans_{\ \!\! \R}(E_{i,i}-E_{i+1,i+1})_{1\leq i \leq n-1}.\ee
where $E_{i,j}$ is the $n\times n$ matrix with $1$ in the $(i,j)$ position and zeroes elsewhere.

\item The orthogonal group: ${\rm O}(n)(\R) = \{A\in GL_n(\R)\mid A^TA = I\}$ and its lie algebra $\mathfrak{o}_n(\R) =\{X\in M_n(\R)\mid F^T+F=0\}$ of skew symmetric matrices

\item The unitary group $U_n(\C)=\{U\in {\rm GL}_n(\C)\mid U^*U=I\}$ and its lie algebra $\mathfrak{u}_n= \{W\in M_n(\C)\mid W+ W^T= 0\}$.

\end{itemize}
The complex lie groups ${\rm GL}_n(\C)$, ${\rm O}_n(\C)$, ${\rm U}(\C)$, ${\rm L}(\C)$, ${\rm D}(\C)$, ${\rm D_1}(\C)$ are defined analogously.

\section{Haar measure on ${\rm SL}_2(\R)$ is Benford in each component}\label{sec:appgeomcone}.

The goal of this section is to provide a geometric proof of Theorem \ref{SLnBenford} in two dimensions. We start with a useful, classical result.

\begin{lem}
The area of the hyperbolic cone \be\{(t,t/x): 0\leq t\leq 1,\ 0< a\leq x\leq b\}\ee is equal to $\log(b/a)$.
\end{lem}

\begin{proof}
The region under the curve $1/x$ has area $\log(b)-\log(a) = \log(b/a)$, and one can form the sector from this region by first attaching the triangle with corners $(0,0),(a,0),(a,1/a)$ and then removing the triangle with corners $(0,0),(b,0),(b,1/b)$. Both triangles have area $1/2$.
\end{proof}

Treating ${\rm SL}_2(\R)$ as the graph of $d=(1-ac)/b$, construct from $A\subset {\rm SL}_n(\R)$ the cone on $A$ to the origin. Since the ${\rm SL}_2(\R)$ action preserves volume, the Haar measure on ${\rm SL}_n(\R)$ equals (up to a scalar) the volume of the cone on $A\subset {\rm SL}_n(\R)$. This observation forms the basis of the proof.

\begin{thm} The $(1,1)$ component of ${\rm SL}_2(\R)$ with Haar measure is $B$-Benford. \end{thm}

\begin{proof}
Write ${\rm SL}_2(\R)$ as \be \left\{\begin{bmatrix} a & b \\ c & d\end{bmatrix} \ : \  ad-bc=1 \right\}.\ee We give a series of statements that simplify the argument but create no loss of generality. Clearly $dg$ is $B$-Benford in the $(1,1)$ component if and only if $c\ dg$ is is $B$-Benford in the $(1,1)$ component, so we take the Haar measure on ${\rm SL}_2(\R)$ that was constructed earlier. Let $a_{11}=a$; notice that $a=0$ is a zero measure subset of $({\rm SL}_2(\R),\mu)$, so we treat ${\rm SL}_2(\R)$ as the graph of the function $d = (bc-1)/a$. By symmetry it suffices to prove the theorem when our sequence of compact sets $K_i$ lie in ${\rm SL}_2(\R)^+$ when $K={\rm graph}(d)$ (with $d=(1-bc)/a$), defined over a rectangular domain $D = [1,x)\times[-\epsilon,\epsilon]\times[-\epsilon,\epsilon]$.

Recall that $\mu(K)=\mu({\rm graph}(d))=\lambda(C({\rm graph}(d)))$ is the volume of the cone consisting of all line segments from $O$ to the graph of $d$. Consider the solid $S := S({\rm graph}(d))$ bounded below the graph of $d$ whose volume is \be \lambda(S) \ = \  \int\!\! \int\!\! \int_D \left(\frac{1-bc}{a}\right) \ da\ db\ dc.\ee

We wish to relate $\lambda(C({\rm graph}(d)))$ to $\lambda(S({\rm graph}(d)))$. By our restriction to positive coordinates, we see that $d$ is decreasing along each ray emanating from the origin in a direction of $D$. As we are assuming ${\rm graph}(d)>0$ on $D$, $\lambda(C({\rm graph}(d))$ can be found by appending to $S$ the three pyramidal regions whose bases are the (3-dimensional) faces of $S$, given by
\be S\cap \{a=1\},S\cap \{b=-\epsilon\},S\cap \{c =-\epsilon\},\ee
then removing the pyramids whose bases are the faces
\be S\cap \{a=x\},S\cap \{b=-\epsilon \},S\cap \{c =\epsilon \}.\ee
The apex for all $6$ pyramids is the origin. Thus\begin{eqnarray}
\lambda (C({\rm graph}(d))) & \ = \ &  \lambda(S) + \lambda(C(S\cap \{a=1\}))-\lambda(C(S\cap \{a=x\})) \nonumber\\
& & \ \ \ \ \ \ \ \ \ + \ \lambda(C(S\cap \{b=-\epsilon\})) - \lambda(C(S\cap \{b=\epsilon\})) \nonumber\\
& & \ \ \ \ \ \ \ \ \ + \ \lambda(C(S\cap \{c =-\epsilon\}))  - \lambda(C(S\cap \{c =\epsilon\})).
\end{eqnarray}

Recall that the $4$-dimensional volume of a pyramid is $1/4$ the volume of the base time the height, and the volume of the base of each pyramid is simply the double integral over the appropriate slice. Thus
\begin{align}
\lambda(C({\rm graph}(d))) \ = \   &\int\!\! \int\!\! \int_D \frac{1-bc}{a} \ da\ db\ dc \nonumber\\
	 & +\ \frac{1}{4}\int_{-\epsilon}^{\epsilon}\!\int_{-\epsilon}^{\epsilon} \frac{1-bc}{1}\ dc\ db  - \frac{x}{4}\int_{-\epsilon}^{\epsilon}\!\int_{-\epsilon}^{\epsilon} \frac{1-bc}{x}\ dc\ db \nonumber\\
 & -\ \frac{\epsilon}{4} \int_{1}^{x}\!\int_{-\epsilon}^{\epsilon} \frac{1+\epsilon c}{a}\ dc\ da - \frac{\epsilon}{4} \int_{1}^{x}\!\int_{-\epsilon}^{\epsilon} \frac{1-\epsilon c}{a}\ dc\ da \nonumber\\
 & -\ \frac{\epsilon}{4} \int_{1}^{x}\!\int_{-\epsilon}^{\epsilon} \frac{1- b\epsilon}{a}\ db\ da  - \frac{\epsilon}{4} \int_{1}^{x}\!\int_{-\epsilon}^{\epsilon} \frac{1- b\epsilon}{a}\ db\ da.
\end{align}

Notice that the second and third terms cancel, and each integral that remains is separable, with the same limits of integration on $a$.
If we let $F(\epsilon)$ be the quantity
\bea F(\epsilon) & \ = \ &  \int_{-\epsilon}^{\epsilon} \int_{-\epsilon}^{\epsilon} (1-bc)\ db\ dc\nonumber\\
               && \ \ +\ \frac{1}{4}\left( \int_{-\epsilon}^{\epsilon}( -\epsilon(1+\epsilon c) -\epsilon(1-\epsilon c))\ dc\right) \nonumber\\
		 & & \ \ + \ \left( \int_{-\epsilon}^{\epsilon}( -\epsilon(1- b\epsilon) - \epsilon(1- b\epsilon))\ db\right), \eea
then
\be \mu(\exp(U_{\epsilon})([0,x)X))\ =\ \lambda(C({\rm graph}(d))) \  = \  \log(x)F(\epsilon),  \ee
and so
\be {\rm Prob}(S_B(a)<x)\ =\ \frac{\log(x)F(\epsilon)}{\log(B)F(\epsilon)}\  =\ \log_B(x). \ee
\end{proof}

\begin{thm}
\label{thm:glndetBenford}
The leading digit law on the determinants of ${\rm GL}_n(\R)$ is $B$-Benford.
\end{thm}

\begin{proof}
Let ${\rm GL}_n(\R)^+$ be the group of all invertible $n\times n$ matrices with positive determinant. The map
\be f\colon {\rm GL}_n(\R)^+\to \R^+\times {\rm SL}_n(\R)\ee
given by $f(g) = (\det (g),(\det (g))^{-\frac{1}{n}} g)$ is a Lie isomorphism, allowing for a decomposition of the Haar measure on  ${\rm {\rm GL}}_n(\R)^+$ as follows: there exists a constant $c>0$ such that for any compactly supported function $\phi\in C_c({\rm GL}_n(\R)^+)$
\be \int_{{\rm GL}_n(\R)^+}\phi(g) \det(g)^{-n}\ dg \ = \  c\int_{\R^+}\frac{dr}{r}\int_{{\rm SL}_n(\R)}\phi(ry)d\mu'(y).\ee

As any compact set in ${\rm GL}_n(\R)^+$ can be well approximated by cubes of the form $[-\epsilon,\epsilon]K'$, $K'\in {\rm SL}_n(\R)$ compact, the result follows.

\end{proof}


\ \\

\end{document}